\documentclass[reqno]{amsart}
\usepackage{amsmath,amsbsy,amsfonts,amsthm}
\usepackage{graphicx}
\usepackage{amscd}
\usepackage{caption}
\usepackage{amssymb}
\usepackage{latexsym}
\usepackage{hyperref}
\usepackage{mathrsfs}
\usepackage[all]{xy}
\usepackage{xcolor}
\theoremstyle{plain}
\usepackage{amssymb}
\usepackage{ragged2e}	
\usepackage{graphicx}
\usepackage{graphics}
\usepackage{lineno}

\theoremstyle{plain}

\newtheorem{example}{\bf Example}
\newtheorem{lemma}{\bf Lemma}

\newtheorem{proposition}{\bf Proposition}
\newtheorem{proposition*}{\bf Proposition}
\newtheorem{remark}{\bf Remark}

\newtheorem{theorem}{\bf Theorem}
\newtheorem*{theorem.a}{\bf Theorem}

\numberwithin{equation}{section}

\begin{document}


\title[A nongradient Ricci almost soliton]
{A nongradient Ricci almost soliton}

\author{Antonio Airton Freitas Filho}
\address{Departamento de Matem\'atica, Universidade Federal do Amazonas, Av. Rodrigo Oct\'avio, 6200, 69080-900 Manaus, Amazonas, Brazil}
\email{aafreitasfilho@ufam.edu.br}
\urladdr{http://www.ufam.edu.br}



\begin{abstract}
In this note, we present a construction method and an explicit example of nongradient (expanding or indefinite) Ricci almost soliton in a warped product. Moreover, we show a rigidity result for the Gaussian soliton.  
\end{abstract}
\maketitle

\section{Introduction}

A well-known generalization of surfaces of revolution are \textit{warped products}; this concept introduced by Bishop and O'Neill in \cite{bis-one}: given two Riemannian manifolds $(B^{k},g_{B}),$ $(F^{m},g_{F})$ and a positive smooth function $w$ in $B,$ define on the product $M^{k+m}=B^{k}\times F^{m}$ the \textit{warped metric}
\begin{align*}
g=\pi^{\ast}g_{B}+(w\circ\pi)^{2}\sigma^{\ast} g_{F},
\end{align*}
where $\pi:M\to B$ and $\sigma:M\to F$ are canonical projection maps. Looking at the warped product $B\times_{w}F:=
(M^{k+m},g),$ we say that $B$ the \textit{base}, $F$ is the \textit{fiber} and $w$ is the \textit{warping function}. When $w$ is constant, we have an usual Riemannian product.

We known that warped products have certain geometric properties for their base and their fiber, for example, the base is a totally geodesic submanifold while the fiber is a umbilical submanifold.

On the study of Einstein metric warped products we suggest the references \cite{BBR,besse,case2,Case,hepeterwylie,kim}. We highlight, for example, the following result:
A warped product $B^{k}\times_{w}F^{m}$ is a Einstein metric, with Einstein constant $c,$ if and only if $(B^{k},g_{B},w)$ is a $m$-quasi-Einstein metric and $(F^{m},g_{F})$ is an Einstein metric, with Einstein constant $\mu,$ where 
\begin{align*}
\mu=c w^{2}+w\Delta_{B} w+(m-1)|\nabla_{B} w|^{2}.
\end{align*}  

The Ricci solitons generalize Einstein metrics and correspond to the self-similar solutions of Ricci flow and often arise as
limits of dilations of singularities in the Ricci flow; see Hamilton \cite{HamiltonRF}. The Ricci almost solitons were introduced as a generalization of Ricci solitons in \cite{Pigola}, having a special family related to the Ricci-Bourguignon flow, see Catino et al. \cite{Catino1} or Catino and Mazzieri \cite{Catino2}.

We say that a complete Riemannian manifold $(M^{n},g)$ is a \textit{Ricci almost soliton} (with \textit{soliton function} $\lambda$), if there exist a smooth vector field $V$ and a smooth function $\lambda$ satisfying
\begin{align*}
Ric + \frac{1}{2}\mathcal{L}_{V}g = \lambda g,    
\end{align*}
where $Ric$ and $\mathcal{L}$ stand, respectively, for the Ricci tensor and the Lie derivative. We shall refer to this equation as the fundamental equation of a Ricci almost soliton $(M^{n},g,V,\lambda).$ It will be called \textit{expanding}, \textit{steady} or \textit{shrinking}, respectively, if $\lambda<0,$ $\lambda=0$ or $\lambda>0.$ Otherwise, it will be called \textit{indefinite}. 

When the vector field $V=\nabla f,$ for some smooth function $f,$ the $(M^{n},g,\nabla f,\lambda)$ will be called a \textit{gradient Ricci almost soliton}. 

Explicit examples of non-Einstein gradient Ricci almost solitons with Codazzi Ricci tensor were presented in \cite{calvi.etal}. Also, some examples of gradient Ricci (almost) solitons warped product were constructed in \cite{RSWP,ARSWP,Gas-Kro,Pigola,romildo,Ivey}. 

For a detailed study of the potential function of gradient Ricci almost soliton warped product we recommend \cite{B-Tenen}.

In \cite{BaDa} $(Sol^{3},V,-2)$ was presented as a nongradient expanding Ricci soliton where $Sol^{3}=B^{2}\times_{e^{-t}}\mathbb{R},$ with $(B^{2},g_{B})=\mathbb{R}\times_{e^{t}}\mathbb{R}$ and $V_{(t,x_{1},x_{2})}=(-2,0,-4x_{2}).$


Our contribution is to explicit example of nongradient (expanding or indefinite) Ricci almost soliton warped product $\mathbb{R}\times_{w}\mathbb{R}^{m}$ (see Example \ref{main2}).

\begin{theorem}\label{main1}
Let $\mathbb{R}\times_{w}\mathbb{R}^{m},$ $m\geqslant 2,$ be a warped product and let 
\begin{align*}
V_{(t,x)}=\left(V_{0}(t,x),V_{1}(t,x),\ldots,V_{m}(t,x)\right)    
\end{align*}
be a smooth vector field. Then $(\mathbb{R}\times_{w}\mathbb{R}^{m},V,\lambda)$ is a Ricci almost soliton if, and only if, 
\begin{align}\label{T1}
V_{(t,x)}=(V_{0}(t),bx_{1}+c_{1},\ldots,bx_{m}+c_{m}),    
\end{align}
\begin{align}\label{T3}
\lambda=-m\dfrac{w^{\prime\prime}}{w}+V_{0}^{\prime},    
\end{align}
\begin{align}\label{T2}
-(m-1)\left(\dfrac{w^{\prime}}{w}\right)^{\prime} - V_{0}\dfrac{w^{\prime}}{w} + V_{0}^{\prime} = b, 
\end{align}
for some constants $b,c_{1},\ldots,c_{m},$ with $b^{2}+c_{1}^{2}+\ldots+c_{m}^{2}\neq0.$ 

In particular, if $w(t)$ is nonconstant smooth function, then $V$ is nongradient smooth vector field.
\end{theorem}

According to the previous theorem, the construction of a Ricci almost soliton on the warped product $\mathbb{R}\times_{w} \mathbb{R}^{m}$ is reduced to finding solutions $V_{0}$ and $w$ to equation \eqref{T2}. 

\begin{remark}[Gaussian soliton rigidity]
With the our theorem it is possible to construct the Gaussian soliton by taking $w(t)=c,$ for some positive constant $c.$ So, by taking $w(t)=c>0,$ we have by equations \eqref{T3} and \eqref{T2} that
\begin{align*}
\lambda=b \ \ \mbox{and} \ \ V_{0}(t)=bt+c_{0},   
\end{align*}
for some constant $c_{0}.$ Hence, $(\mathbb{R}^{1+m},g_{\circ},V,b)$ is a Gaussian soliton, with $V=Df,$ where $f:\mathbb{R}^{1+m}\to\mathbb{R}$ is given by
\begin{align*}
f(t,x_{1},\ldots,x_{m}) = \dfrac{b}{2}\|(t,x_{1},\ldots,x_{m})\|^{2} + (t,x_{1},\ldots,x_{m})\cdot(c_{0},c_{1},\ldots,c_{m}) +\kappa,
\end{align*}
for some constant $\kappa.$
\end{remark}

\section{Main Lemma}

In this section, we consider $(t,x)\in\mathbb{R}\times\mathbb{R}^{m},$ $m\geqslant2,$ with coordinate system $(t,x_{1},\ldots,x_{m}).$ Here $\mathbb{R}\times_{w}\mathbb{R}^{m}$ denotes the Riemannian manifold $(\mathbb{R}^{1+m},g)$ where 
\begin{align*}
g=dt^{2}+w^{2}\sum\limits_{i=1}^{m}dx_{i}^{2}.    
\end{align*} 

The following lemma will be useful in the proof of our main result.

\begin{lemma}\label{Lem.1}
Let $\mathbb{R}\times_{w}\mathbb{R}^{m},$ $m\geqslant 2,$ be a warped product and let
$V=(V_{0},V_{1},\ldots,V_{m})$ be a smooth vector field. Then
\begin{align}\label{L1}
Ric = -m\dfrac{w^{\prime\prime}}{w}dt^{2} - \left(\dfrac{w^{\prime\prime}}{w} + (m-1)\left(\dfrac{w^{\prime}}{w}\right)^{2}\right) w^{2}\sum_{i=1}^{m}dx_{i}^{2},
\end{align}
\begin{align}\label{L2}
\nonumber\dfrac{1}{2}\mathcal{L}_{V}g &= \dfrac{\partial V_{0}}{\partial t} dt^{2} + \sum_{i=1}^{m}\left(V_{0}ww^{\prime} + w^{2} \dfrac{\partial V_{i}}{\partial x_{i}}\right) dx_{i}^{2}\\ 
&\ \ \ +\dfrac{1}{2}\sum_{i=1}^{m} \left(\dfrac{\partial V_{0}}{\partial x_{i}} + w^{2} \dfrac{\partial V_{i}}{\partial t} \right)\big(dt\otimes dx_{i} +dx_{i} \otimes dt\big)\\
\nonumber&\ \ \ + \dfrac{w^{2}}{2} \sum_{i\neq j} \dfrac{\partial V_{i}}{\partial x_{j}} \big(dx_{j}\otimes dx_{i} + dx_{i} \otimes dx_{j}\big).    
\end{align}
\end{lemma}
\begin{proof}
We begin by using Lemma $7.4$ in \cite{bis-one}, for the warped product $\mathbb{R}\times_{w}\mathbb{R}^{m},$ so the equation \eqref{L1} holds.

To verify the equation \eqref{L2}, we calculate the Lie derivative of the metric $g$ in the $V$ direction, so  
\begin{align*}
\mathcal{L}_{V}g &= \mathcal{L}_{V}(dt^{2})+V(w^{2})\sum\limits_{i=1}^{m}dx_{i}^{2} + w^{2}\sum\limits_{i=1}^{m}\mathcal{L}_{V}\left(dx_{i}^{2}\right)\\
&= dV(t)\otimes dt + dt\otimes dV(t) + 2V_{0}ww^{\prime }\sum\limits_{i=1}^{m}dx_{i}^{2} \\
&\ \ \ + w^{2}\sum\limits_{i=1}^{m}\left(dV(x_{i})\otimes dx_{i} + dx_{i}\otimes dV(x_{i})\right)\\
&= dV_{0} \otimes dt + dt \otimes dV_{0} + 2V_{0}ww^{\prime }\sum\limits_{i=1}^{m}dx_{i}^{2}\\
&\ \ \ + w^{2}\sum\limits_{i=1}^{m}\left(dV_{i}\otimes dx_{i} + dx_{i}\otimes dV_{i}\right)\\
&= \dfrac{\partial V_{0}}{\partial t}\left(dt \otimes dt + dt \otimes dt\right) + \sum\limits_{i=1}^{m}\dfrac{\partial V_{0}}{\partial x_{i}}\left(dx_{i} \otimes dt + dt \otimes dx_{i}\right)\\
&\ \ \ + 2V_{0}ww^{\prime }\sum\limits_{i=1}^{m}dx_{i}^{2} + w^{2} \sum\limits_{i=1}^{m}\dfrac{\partial V_{i}}{\partial t}\left(dt \otimes dx_{i} + dx_{i} \otimes dt\right)\\
&\ \ \ + w^{2} \sum\limits_{i,j=1}^{m}
\dfrac{\partial V_{i}}{\partial x_{j}}\left(dx_{j} \otimes dx_{i} + dx_{i} \otimes dx_{j}\right).
\end{align*}

This concludes the proof of lemma.
\end{proof}

\section{Proof of Theorem \ref{main1}}

\begin{proof}[\textbf{Proof of Theorem~\ref{main1}}]
We begin using the fundamental equation of Ricci almost soliton, we deduce from Lemma \ref{Lem.1} that 
\begin{itemize}
\item[$\bullet$] \underline{the terms} $dt^{2}:$
\begin{align}\label{eqA}
-m\dfrac{w^{\prime\prime}}{w} + \dfrac{\partial V_{0}}{\partial t} = \lambda;   
\end{align}
\item[$\bullet$]\underline{each term} $dx_{i}^{2}:$ 
\begin{align}\label{eqB}
-\dfrac{w^{\prime\prime}}{w} -(m-1)\left(\dfrac{w^{\prime}}{w}\right)^{2} + V_{0}\dfrac{w^{\prime}}{w} +\dfrac{\partial V_{i}}{\partial x_{i}} = \lambda;    
\end{align}
\item[$\bullet$]\underline{each term} $dt\otimes dx_{i}+dx_{i}\otimes dt:$ 
\begin{align}\label{eqC}
\dfrac{\partial V_{0}}{\partial x_{i}} + w^{2} \dfrac{\partial V_{i}}{\partial t} = 0;   
\end{align}
\item[$\bullet$]\underline{each term} $dx_{i}\otimes dx_{j}+dx_{j}\otimes dx_{i},$ with $i\neq j:$  
\begin{align}\label{eqD}
\dfrac{\partial V_{i}}{\partial x_{j}}=0;    
\end{align}
\end{itemize}
for $i,j=1,\ldots,m.$

So, from equation \eqref{eqD},
\begin{align*}
V_{i}(t,x) = V_{i}(t,x_{i}), \ \ i=1,\ldots,m,    
\end{align*}
and by equation \eqref{eqB}, for $i\neq j,$ we get
\begin{align*}
\dfrac{\partial V_{i}}{\partial x_{i}}(t,x_{i}) = \dfrac{\partial V_{j}}{\partial x_{j}}(t,x_{j}),  
\end{align*}
i.e.
\begin{align*}
\dfrac{\partial V_{i}}{\partial x_{i}}(t,x_{i}) = b(t),
\end{align*}
which implies, by integrating
\begin{align*}
V_{i}(t,x_{i}) = b(t)x_{i} + c_{i}(t),    
\end{align*}
for some smooth functions $b(t),c_{1}(t),\ldots,c_{m}(t).$

Now, for each $i=1,\ldots,m,$ follows from equation \eqref{eqC} that
\begin{align}\label{V0}
\nonumber V_{0}(t,x) &= \int^{x_{i}}\dfrac{\partial V_{0}}{\partial x_{i}}(t,s) ds + a_{i}(t,x_{1},\ldots,x_{i-1},x_{i+1},\ldots,x_{m})\\
\nonumber&= -w(t)^{2}\int^{x_{i}}\dfrac{\partial V_{i}}{\partial t}(t,s) ds + a_{i}(t,x_{1},\ldots,x_{i-1},x_{i+1},\ldots,x_{m})\\
\nonumber&= -w(t)^{2}\int^{x_{i}} \left(b^{\prime}(t)s+c_{i}^{\prime}(t)\right)ds + a_{i}(t,x_{1},\ldots,x_{i-1},x_{i+1},\ldots,x_{m})\\
&= -w(t)^{2} \left(b^{\prime}(t)\dfrac{x_{i}^{2}}{2} + c_{i}^{\prime}(t)x_{i} \right) + a_{i}(t,x_{1},\ldots,x_{i-1},x_{i+1},\ldots,x_{m}),
\end{align}
where $a_{i}(t,x_{1},\ldots,x_{i-1},x_{i+1},\ldots,x_{m})$ are smooth functions.  

Since, for all $i\neq j,$ from \eqref{V0} and denoting $a_{i}(\widehat{x_{i}})=a_{i}(t,x_{1},\ldots,x_{i-1},x_{i+1},\ldots,x_{m})$ we get  
\begin{align*}
a_{i}(\widehat{x_{i}}) + w(t)^{2} \left(b^{\prime}(t)\dfrac{x_{j}^{2}}{2} + c_{j}^{\prime}(t)x_{j} \right) = a_{j}(\widehat{x_{j}}) + w(t)^{2} \left(b^{\prime}(t)\dfrac{x_{i}^{2}}{2} + c_{i}^{\prime}(t)x_{i} \right),  
\end{align*}
so
\begin{align*}
b^{\prime}(t)=c_{i}^{\prime}(t)=c_{j}^{\prime}(t)=0 \ \ \ a_{i}(\widehat{x_{i}}) = a_{j}(\widehat{x_{j}}).   
\end{align*}

Therefore
\begin{align*}
V_{(t,x)} = \left(V_{0}(t),bx_{1}+c_{1},\ldots,bx_{m}+c_{m}\right),   
\end{align*}
for some constants $b,c_{1},\ldots,c_{m}.$

It is clear that \eqref{T3} holds from equation \eqref{eqA}, and \eqref{T2} follows equations \eqref{eqA} and \eqref{eqB} by a straightforward calculation. 

Conversely, if \eqref{T1}, \eqref{T3} and \eqref{T2} hold, we obtain  by Lemma \ref{Lem.1}, 
\begin{align*}
Ric + \dfrac{1}{2}\mathcal{L}_{V}g &= \left(-m\dfrac{w^{\prime\prime}}{w}+V_{0}^{\prime}\right)dt^{2}\\
& \ \ \ \ +\left(-\dfrac{w^{\prime\prime}}{w}-(m-1)\left(\dfrac{w^{\prime}}{w}\right)^{2}+V_{0}\dfrac{w^{\prime}}{w}+b\right)w^{2}\sum\limits_{i=1}^{m}dx_{i}^{2}\\
&=\lambda dt^{2} + \left(-\dfrac{w^{\prime\prime}}{w}-(m-1)\left(\dfrac{w^{\prime}}{w}\right)^{2}-(m-1)\left(\dfrac{w^{\prime}}{w}\right)^{\prime}+V_{0}^{\prime}\right)w^{2}\sum\limits_{i=1}^{m}dx_{i}^{2}\\
&=\lambda dt^{2} + \left(-m\dfrac{w^{\prime\prime}}{w}+V_{0}^{\prime}\right)w^{2}\sum\limits_{i=1}^{m}dx_{i}^{2}\\
&=\lambda dt^{2} + \lambda w^{2}\sum\limits_{i=1}^{m}dx_{i}^{2}=\lambda \left(dt^{2} + w^{2}\sum\limits_{i=1}^{m}dx_{i}^{2}\right) = \lambda g.
\end{align*}

In particular, if $w(t)$ is nonconstant,  from equation \eqref{T1} we have
\begin{align*}
V^{\flat} &= g(V,\cdot) = dt(V)dt + w^{2}\sum\limits_{i=1}^{m}dx_{i}(V)dx_{i}\\
&= V_{0}dt + w^{2}\sum\limits_{i=1}^{m} V_{i}dx_{i} = V_{0}dt + w^{2}\sum\limits_{i=1}^{m}\left(bx_{i}+c_{i}\right)dx_{i},
\end{align*}
which implies, by exterior derivative
\begin{align*}
dV^{\flat} = 2ww^{\prime}\sum\limits_{i=1}^{m}(bx_{i}+c_{i})dt\wedge dx_{i} \neq 0,   
\end{align*}
i.e., $V$ is nongradient smooth vector field.

This complete the proof of the theorem.
\end{proof}

\section{A Nongradient Ricci Almost Soliton}

Now, with the Theorem \ref{main1} in mind, we consider $b\neq0$ and $V_{0}(t)=a,$ for some nonzero constant $a.$ 

Firstly, from equation \eqref{T2}, we obtain
\begin{align*}
\left(\dfrac{w^{\prime}}{w}\right)^{\prime} +\dfrac{a}{m-1}\dfrac{w^{\prime}}{w} = -\dfrac{b}{m-1},    
\end{align*}
which reduces to
\begin{align*}
\left(\exp\!\left(\dfrac{at}{m-1}\right)\dfrac{w^{\prime}}{w}\right)^{\prime} = -\dfrac{b}{m-1}\cdot\exp\!\left(\dfrac{at}{m-1}\right),   
\end{align*}
so, by integrating
\begin{align*}
\dfrac{w^{\prime}}{w} = -\dfrac{b}{a}+c\cdot\exp\!\left(-\dfrac{at}{m-1}\right),    
\end{align*}
for some constant $c\neq0.$

Again, by integrating
\begin{align*}
w(t) = K\cdot \exp\!\left(\!-\frac{(m-1)c}{a}\cdot\exp\!\left(\!-\frac{at}{m-1}\right) -\frac{bt}{a}\right),
\end{align*}
for some constant $K>0.$

Secondly, from equation \eqref{T3}, we get
\begin{align*}
\lambda(t) = -mc^{2}\cdot\left(\exp\!\left(\!-\frac{at}{m-1}\right)-\frac{b}{ac}\right)^{2}
+ \frac{mac}{m-1}\cdot\exp\!\left(\!-\frac{at}{m-1}\right).
\end{align*}

The next proposition provides some properties of the soliton function $\lambda.$  

\begin{proposition}\label{Prop.1}
Let $a,b,c$ be constants and $m\geqslant2.$ 
The smooth function $\lambda:\mathbb{R}\to\mathbb{R}$ given by
\begin{align*}
\lambda(t) = -mc^{2}\cdot\left(\exp\!\left(\!-\frac{at}{m-1}\right)-\frac{b}{ac}\right)^{2}
+ \frac{mac}{m-1}\cdot\exp\!\left(\!-\frac{at}{m-1}\right),
\end{align*}
satisfies
\begin{itemize}
\item[$(i)_{A}$] \underline{For $a<0:$} 
\begin{align*}
\lim\limits_{t\to-\infty} \lambda(t) = -m\frac{b^{2}}{a^{2}}<0 \ \ \ \mbox{and} \ \ \ \lim\limits_{t\to+\infty} \lambda(t) = -\infty.
\end{align*}
\item[$(i)_{B}$] \underline{For $a>0:$} 
\begin{align*}
\lim\limits_{t\to-\infty} \lambda(t) = -\infty \ \ \ \mbox{and} \ \ \ \lim\limits_{t\to+\infty} \lambda(t) = -m\frac{b^{2}}{a^{2}}<0. 
\end{align*}
\item[$(ii)_{A}$] If $\dfrac{a^{2}+2(m-1)b}{ac}>0,$ then $\lambda$ reaches the unique maximum value
\begin{align*}
\lambda(t_{0}) = \frac{m}{4(m-1)^{2}}\big[a^{2} +4(m-1)b\big],
\end{align*}
where $t_{0}=\ln\left(\dfrac{a^{2}+2(m-1)b}{2(m-1)ac}\right)^{-\frac{(m-1)}{a}}.$
\\
\item[$(ii)_{B}$] Otherwise, $\lambda$ is negative and strictly monotone.
\end{itemize}
\end{proposition}
\begin{proof}
Note that the itens $(i)_{A}$ and $(ii)_{B}$ are immediately verified. 

For the itens $(ii)_{A}$ and $(ii)_{B}$ we must calculate the first and second order derivatives of $\lambda.$ So, by a straightforward calculation we get
\begin{align*}
\lambda^{\prime}(t) =
 & \ \frac{2mac^{2}}{m-1}\cdot\exp\!\left(\!-\frac{2at}{m-1}\right)\\
&-\frac{mc}{(m-1)^{2}}\left(a^{2}+2(m-1)b\right)\cdot\exp\!\left(\!-\frac{at}{m-1}\right)
\end{align*}
and 
\begin{align*}
\lambda^{\prime\prime}(t) = &-\frac{4ma^{2}c^{2}}{(m-1)^{2}}\cdot\exp\!\left(\!-\frac{2at}{m-1}\right)\\
&-\frac{mac}{(m-1)^{3}}\left(ma^{2}+2(m-1)b\right)\cdot\exp\!\left(\!-\frac{at}{m-1}\right). 
\end{align*}
Hence, if $\frac{a^{2}+2(m-1)b}{ac}>0,$ then $\lambda$ has a unique maximum point in 
\begin{align*}
t_{0}=\ln\left(\frac{a^{2}+2(m-1)b}{2(m-1)ac}\right)^{-\frac{(m-1)}{a}}, 
\end{align*}
with maximum value 
\begin{align}\label{val}
\lambda(t_{0}) = \frac{m}{4(m-1)^{2}}\left(a^{2} +4(m-1)b\right),
\end{align}
since holds either $(i)_{A}$ or $(i)_{B},$
and
\begin{align*}
\lambda^{\prime}(t_{0}) = 0 \ \ \ \mbox{and} \ \ \ \lambda^{\prime\prime}(t_{0}) = -\frac{m^{2}}{(m-1)^{4}}c^{2}\left(a^{2}+2(m-1)b\right)^{2}<0.
\end{align*}
\end{proof}

\begin{remark}
Note that, if holds \eqref{val}, then we can chosen the signal of $a^{2}+4(m-1)b$ even with $\dfrac{a^{2}+2(m-1)b}{ac}>0.$ Hence, we obtain a nongradient (expanding or indefinite) Ricci almost soliton $(\mathbb{R}\times_{w}\mathbb{R}^{m},V,\lambda).$    
\end{remark}

Finally, we present our example of nongradient (expanding or indefinite) Ricci almost soliton.

\begin{example}[A nongrandient Ricci almost soliton]\label{main2}
Let $\mathbb{R}\times_{w}\mathbb{R}^{m},$ $m\geqslant 2,$ be a warped product, where
\begin{align*}
w(t) = K\cdot \exp\!\left(\!-\frac{(m-1)c}{a}\cdot\exp\!\left(\!-\frac{at}{m-1}\right) -\frac{bt}{a}\right),
\end{align*}
for some nonzero constants $a,b,c$ and $K>0.$ Consider the nongradient smooth vector field
\begin{align*}
V_{(t,x)}=\left(a,bx_{1}+c_{1},\ldots,bx_{m}+c_{m}\right),\end{align*}
for some constants $c_{1},\ldots,c_{m}$ and the smooth function
\begin{align*}
\lambda(t) = -mc^{2}\cdot\left(\exp\!\left(\!-\frac{at}{m-1}\right)-\frac{b}{ac}\right)^{2}
+ \frac{mac}{m-1}\cdot\exp\!\left(\!-\frac{at}{m-1}\right).
\end{align*}
Then, by Theorem \ref{main1} and Proposition \ref{Prop.1}, we conclude that $(\mathbb{R}\times_{w}\mathbb{R}^{m},V,\lambda)$ is a nongradient (expanding or indefinite) Ricci almost soliton.
\end{example}

\begin{remark}
Recall that the gradienticity is a Riemannian metric condition. In fact, for the standard metric of $\mathbb{R}^{1+m},$
\begin{align*}
g_{\circ} = dt^{2} + \sum\limits_{i=1}^{m}dx_{i}^{2},   
\end{align*}
the vector field $V_{(t,x)}=(a,bx_{1}+c_{1},\ldots,bx_{m}+c_{m})$ is the gradient of smooth function $u:\mathbb{R}^{1+m}\to\mathbb{R}$ given by
\begin{align*}
u(t,x_{1},\ldots,x_{m}) = at+\dfrac{b}{2}\|(x_{1},\ldots,x_{m})\|^{2} + (x_{1},\ldots,x_{m})\cdot(c_{1},\ldots,c_{m}) + \kappa,    
\end{align*}
for some constant $\kappa.$
\end{remark}

\vspace{0.3cm}
\textbf{Acknowledgements:} The author would like to express his gratitude to the anonymous reviewer for his careful reading and for his corrections and constructive suggestions to the text. Thanks to your work, we were able to improve this article.

\vspace{0.3cm}
\textbf{Data availability:} No data was used for the research described in the article.

\end{document}